\documentclass{article}
\usepackage[utf8]{inputenc}
\usepackage{amsmath,amsthm,fullpage,graphicx}
\usepackage{hyperref}
\usepackage{listings}
\usepackage{authblk}
\usepackage{url}
\usepackage{pdflscape}
\newtheorem{theorem}{Theorem}

\newtheorem{corollary}{Corollary}

\newtheorem{definition}{Definition}
\newtheorem{conjecture}{Conjecture}

\title{Pandigital and penholodigital numbers}
\author{Chai Wah Wu}
\affil{IBM Research\\IBM T. J. Watson Research Center, Yorktown Heights, NY, USA\thanks{cwwu@us.ibm.com}}

\date{March 29, 2024\\Latest update: March 13, 2025}

\begin{document}

\maketitle

\begin{abstract} Pandigital and penholodigital numbers are numbers that contain every digit or nonzero digit respectively.  We study properties of pandigital or penholodigital numbers that are also square, oblong or prime.
\end{abstract}

\section{Introduction}
Pandigital \cite{wiki:Pandigital_number} and penholodigital numbers are defined as numbers that contain every digit or every nonzero digit, respectively. More precisely,
\begin{definition}
A number $n$ is a {\em pandigital} number in base $b$ if $n$ expressed as a base $b$ number contains each of the $b$ different digits at least once. A number $n$ is a {\em strict pandigital} number in base $b$ if $n$ expressed as a base $b$ number contains each of the $b$ different digits exactly once.
\end{definition}
A strict pandigital number is pandigital and clearly there are a finite number of strict pandigital numbers for each base $b$. 
\begin{definition}
A number $n$ is a {\em penholodigital} number in base $b$ if $n$ expressed as a base $b$ number does not contain the zero digit\footnote{Such numbers are also called zeroless numbers.} and contains each of the $b-1$ different nonzero digits at least once. A number $n$ is a {\em strict penholodigital number} in base $b$ if $n$ expressed as a base $b$ number does not contain the zero digit and contains each of the $b-1$ different nonzero digits exactly once.
\end{definition}

For example, in base $10$, $1023456798$ is a strict pandigital number and $10023546789$ is a pandigital number.
Similarly, $123456798$ is a strict penholodigital number and $1323546789$ is a penholodigital number in base $10$.

Many of these numbers are listed as sequences in the On-line Encyclopedia of Integer Sequences (OEIS) \cite{oeis}.
Examples of pandigital and strict pandigital numbers in base $10$ are listed in OEIS sequences  \href{https://oeis.org/A171102}{A171102} and  \href{https://oeis.org/A050278}{A050278} respectively. Examples of penholodigital numbers in base $10$ are listed in OEIS sequence  \href{https://oeis.org/A050289}{A050289}.
The smallest and largest strict pandigital numbers in base $b$ are listed in OEIS sequences  \href{https://oeis.org/A049363}{A049363} and  \href{https://oeis.org/A062813}{A062813} respectively.
The smallest and largest strict penholodigital numbers in base $b$ are listed in OEIS sequences  \href{https://oeis.org/A023811}{A023811} and  \href{https://oeis.org/A051846}{A051846} respectively.

\section{Pandigital and penholodigital squares}

Let $s_b(n)$ be the sum of the digits of $n$ in base $b$. Since $b\equiv 1 \pmod{b-1}$ this means that $b^k \equiv 1 \pmod{b-1}$, which in turn implies that $s_b(n)\equiv n \pmod
{b-1}$.

Note that for a strict pandigital or a strict penholodigital number $n$,
$s_b(n) = b(b-1)/2$.
This implies directly the following:

\begin{theorem} \label{thm:squares}
Let $A_b$ be the set of modular square roots of $b(b-1)/2$ modulo $b-1$, i.e. it is the set of integers $0\leq m < b-1$ such that $m^2 \equiv b(b-1)/2 \pmod{b-1}$.
If $n^2$ is a strict pandigital or a strict penholodigital square, then $n\equiv m \pmod{b-1}$ for some $m\in A_b$. 
\end{theorem}

For instance, $A_{10} = \{0,3,6\}$. Thus if $n^2$ is a strict pandigital or a strict penholodigital square in base $10$, then $n\equiv 0 \pmod{3}$. In general, finding square roots modulo $m$ is a difficult problem, as difficult as factoring $m$ \cite{granville:gauss}, but for some values of $b$, we can explicitly find $A_b$.

\begin{theorem} \label{thm:setA}
$b$ is odd and $b-1$ has an even $2$-adic valuation\footnote{The $2$-adic valuation of $n$ is the largest power of $2$ that divides $n$.} if and only if $A_b = \emptyset$. If $b$ is an even squarefree number, then $A_b = \{0\}$. If $b$ is an odd squarefree number, then $A_b=\{(b-1)/2\}$.
\end{theorem}

\begin{proof}
If $b$ is even, $b(b-1)/2 \equiv 0 \pmod{b-1}$ and $0\in A_b$.
Let $b=2^{2k+1}q+1$ for some odd $q$. Then $b-1 = 2^{2k+1}q$,  $b(b-1)/2\equiv 2^{2k}q \pmod{b-1}$ and $(2^kq)^2 + \frac{1-q}{2}2^{2k+1}q = 2^{2k}q$, i.e.  $(2^kq)^2 \equiv 2^{2k}q \pmod{2^{2k+1}q}$ and thus $2^kq \in A_b$. So in both these cases $A_b\neq \emptyset$. 

Now suppose $b = 2^{2k}q+1$ for some $k>0$ and odd $q$. First note that $b-1 = 2^{2k}q$ and $2^{2k}$ and $q$ are coprime. Since $b-1$ is even, $b(b-1)/2\equiv (b-1)/2 \pmod{b-1}$. Let $r = (b-1)/2 = 2^{2k-1}q$. If $k=1$, $b-1 = 4q$, $r = 2q$ and thus $r\equiv 2 \pmod{4}$ is a quadratic nonresidue modulo $4$. By the Chinese Remainder Theorem, $r$ is a quadratic nonresidue modulo $b-1$.
If $k>1$, Gauss showed \cite{Gauss1801,granville:gauss} that a nonzero number $r$ is a residue modulo $2^{2k}$ if and only if $r$ is of the form $2^{2m}(8j+1)$. Since $r = 2^{2k-1}q$, $r$ is a quadratic nonresidue modulo $2^{2k}$. Thus if $b$ is odd and $b-1$ has an even $2$-adic valuation, then $A_b=\emptyset$.

Next, suppose $b-1$ is odd and squarefree.  Then $b$ is even, $b-1$ divides $b(b-1)/2$ and $m^2 \equiv 0 \pmod{b-1}$. Since $b-1 = \prod_i p_i$ for distinct odd primes $p_i$, and $m^2 \equiv 0 \pmod{p_i}$ if and only if $m \equiv 0 \pmod{p_i}$, this implies that $m \equiv 0 \pmod{b-1}$ by the Chinese Remainder Theorem, i.e. $A_b = \{0\}$.
Similarly, if $b-1$ is even and squarefree, then $m^2 \equiv b(b-1)/2 \equiv (b-1)/2 \pmod{b-1}$ and $(b-1)/2 = \prod_i p_i$ for distinct odd primes $p_i$, i.e., $m^2 \equiv 0 \pmod{p_i}$ and $m^2 \equiv 1 \pmod{2}$. This implies that $m \equiv 0 \pmod{p_i}$ and $m \equiv 1 \pmod{2}$. Again by the Chinese Remainder Theorem, $m \equiv (b-1)/2 \pmod{b-1}$ and $A_b=\{(b-1)/2\}$.
\end{proof}

Theorems \ref{thm:squares} and \ref{thm:setA} result in the following immediate consequences.

\begin{corollary} \label{cor:squares}
If $b$ is odd and $b-1$ has an even $2$-adic valuation, then there are no strict pandigital nor strict penholodigital squares in base $b$.
\end{corollary}

Corollary \ref{cor:squares} for the case of pandigital squares was also shown in \cite{partridge:pandigital:2015} directly using a different technique. 
However, as we will see in Section \ref{sec:oblong}, the approach in this section by means of the set $A_b$ allows us to extend this result easily to other numbers beyond squares, such as oblong numbers. We conjecture the following:

\begin{conjecture}
Suppose $b>4$.
A strict pandigital square and a strict penholodigital square in base $b$ exists
if and only if $b$ is even or $b-1$ has an odd $2$-adic valuation.
\end{conjecture}

\begin{corollary}
Let $m^2$ be a strict pandigital or a strict penholodigital square in base $b$.
If $b$ is an even squarefree number, then
$m\equiv 0 \pmod{b-1}$.
If $b$ is an odd squarefree number, then
$m\equiv (b-1)/2 \pmod{b-1}$.
\end{corollary}

The number of strict pandigital and strict penholodigital squares for each base $b$ are listed in OEIS sequences  \href{https://oeis.org/A258103}{A258103}
and  \href{https://oeis.org/A370950}{A370950} respectively.

\section{Pandigital and penholodigital oblong numbers} \label{sec:oblong}

Similarly to squares, for oblong (or pronic) numbers (i.e. numbers of the form $m(m+1)$), we have the following:
\begin{theorem}\label{thm:oblong}
Let $B_b$ be the set of numbers $m$ such that $m(m+1) \equiv b(b-1)/2 \pmod{b-1}$. If $n(n+1)$ is strict pandigital or strict penholodigital,
then $n\equiv m \pmod{b-1}$ for some $m\in B_b$.
\end{theorem}

The following result shows when $B_b$ is empty.
\begin{theorem}
$b \equiv 3 \pmod{4}$ if and only if $B_b=\emptyset$. 
If $b$ is even, then $0\in B_b$. If $b\equiv 1\pmod{4}$, then $(b-1)/2\in B_b$.
\end{theorem}
\begin{proof}
Suppose $b=2q$ is even, then $b(b-1)/2 = q(b-1) \equiv 0 \pmod{b-1}$ and thus $0\in B_b$.
Suppose $b=4q+1$, then $b-1=4q$ and $b(b-1)/2 = (4q+1)2q \equiv 2q \pmod{b-1}$. Since $2q(2q+1)\equiv 2q \pmod{b-1}$, this means that $(b-1)/2 = 2q \in B_b$.
Suppose $b\equiv 3 \pmod{4}$, i.e. $b = 2q+1$ for some odd number $q$. Then $b(b-1)/2 = 2q^2+q$. Since $b-1 = 2q$, this means that
$b(b-1)/2\equiv q \pmod{b-1}$. Since $q$ is odd, $m(m+1)$ is even and $b-1$ is even, this means that $m(m+1)\not\equiv q \pmod{b-1}$ and
thus $B_b=\emptyset$.
\end{proof}

\begin{corollary} \label{cor:oblong}
If $b \equiv 3 \pmod{4}$, then there are no strict pandigital nor strict penholodigital oblong numbers in base $b$.
\end{corollary}

\begin{conjecture}
Suppose $b>5$.
A strict pandigital oblong number and a strict penholodigital oblong number in base $b$ exists
if and only if $b\not\equiv 3\pmod{4}$.
\end{conjecture}

The roots of the smallest strict pandigital and strict penholodigital oblong number in base $b$ are given by OEIS sequence  \href{https://oeis.org/A381266}{A381266} and \href{https://oeis.org/A382050}{A382050} respectively.

\section{Pandigital and penholodigital primes} \label{sec:prime}
Since $s_b(n) = b(b-1)/2$, this means that $n\equiv 0 \pmod{b-1}$ if $b$ is even and
$n\equiv 0 \pmod{(b-1)/2}$ if $b$ is odd. This implies that there are no strict pandigital nor strict penholodigital prime numbers in base $b> 3$, i.e. a pandigital prime must be larger or equal to $\frac{b^b-b^2+b-1}{(b-1)^2}+b^b$ (i.e. the base $b$ representation is $10123....(b-1)$) and
a penholodigital prime must be larger or equal to $\frac{b^b-b^2+b-1}{(b-1)^2}+b^{b-1}$  (i.e. the base $b$ representation is $1123....(b-1)$). In other words, we have the following lower bounds:

\begin{theorem} \label{thm:prime}
Let $b>3$. If $n$ is a pandigital prime in base $b$, then $n\geq\frac{b^b-b^2+b-1}{(b-1)^2}+b^b$. If $n$ is a penholodigital prime in base $b$, then $n\geq  \frac{b^b-b^2+b-1}{(b-1)^2}+b^{b-1}$.
\end{theorem}

These lower bounds can be improved for bases of the form $b=4k+3$.

\begin{theorem} \label{thm:prime2}
If $b = 4k+3$ for $k>0$, then a pandigital prime in base $b$ is larger than or equal to $n\geq\frac{b^b-b^2+b-1}{(b-1)^2}+b^b+b^{b-2}$ and a penholodigital prime in base $b$ is larger than or equal to $n\geq\frac{b^b-b^2+b-1}{(b-1)^2}+b^{b-1}+b^{b-2}$.
\end{theorem} 

\begin{proof}
For a pandigital prime, if $b=4k+3$, then $b(b-1)/2+1 = 2(4k^2+5k+2)$  and $b-1$ are both even. Thus if $s_b(n) =  b(b-1)/2+1$, then $n>2$ is even and thus not prime. Thus $s_b(n)\geq b(b-1)/2+2$ and thus
$n$ is larger than or equal to $10223....(n-1)$ in base $b$.

Similarly, for a penholodigital prime, $n$ is larger than or equal to $1223....(n-1)$ in base $b$.
\end{proof}

The smallest pandigital and penholodigital primes are listed in OEIS sequences  \href{https://oeis.org/A185122}{A185122} and  \href{https://oeis.org/A371194}{A371194} respectively. Numerical experiments suggest the following conjecture:
\begin{conjecture} 
For $b>3$, the smallest pandigital prime or penholodigital prime $n$ satisfy $s_b(n) = b(b-1)/2+2$ if $b$ is of the form $4k+3$ and  satisfy $s_b(n) = b(b-1)/2+1$ otherwise. 
\end{conjecture}

\section{Subpandigital and subpenholodigital numbers}
We can also consider numbers whose digits in base $b$ include all (nonzero) digits up to $b-2$.

\begin{definition}
A number $n$ is a {\em subpandigital} number in base $b$ if $n$ expressed as a base $b$ number does not contain the digit $b-1$ and contains each of the $b-1$ digits $0,1,\cdots, b-2$ at least once. A number $n$ is a {\em strict subpandigital} number in base $b$ if $n$ expressed as a base $b$ number does not contain the digit $b-1$ and contains each of the $b-1$ digits $0,1,\cdots, b-2$ exactly once.
\end{definition}
\begin{definition}
A number $n$ is a {\em subpenholodigital} number in base $b$ if $n$ expressed as a base $b$ number does not contain the zero digit nor the digit $b-1$ and contains each of the $b-2$ digits $1,\cdots, b-2$ at least once. A number $n$ is a {\em strict subpenholodigital number} in base $b$ if $n$ expressed as a base $b$ number does not contain the zero digit nor the digit $b-1$ and contains each of the $b-2$ digits $1,\cdots, b-2$ exactly once.
\end{definition}

For example, in base $10$, $120345687$ is a strict subpandigital number as it contains all digits except $9$ exactly once and $87654123$ is a strict subpenholodigital number as it contains all nonzero digits except $9$ exactly once.
Since there are no subpenholodigital number in base $2$ and the only subpandigital number in base $2$ is $0$, we only consider bases $b>2$ in this section. 
As $s_b(n) = (b-1)(b-2)/2 = b(b-1)/2 - (b-1)$ for a strict subpandigital or a strict subpenholodigital number $n$ and thus $(b-1)(b-2)/2 \equiv b(b-1)/2 \pmod{b-1}$, we have the following analog result to Theorem \ref{thm:squares}:
\begin{theorem} \label{thm:squares:sub}
Let $A_b$ be as defined in Theorem \ref{thm:squares}.
If $n^2$ is a strict subpandigital or a strict subpenholodigital square, then $n\equiv m \pmod{b-1}$ for some $m\in A_b$. 
\end{theorem}
\begin{corollary} \label{cor:squares':sub}
If $b$ is odd and $b-1$ has an even $2$-adic valuation, then there are no strict subpandigital nor strict subpenholodigital squares in base $b$.
\end{corollary}

Similarly, we conjecture the following:
\begin{conjecture}
Suppose $b>7$.
A strict subpandigital square and a strict subpenholodigital square in base $b$ exists
if and only if $b$ is even or $b-1$ has an odd $2$-adic valuation.
\end{conjecture}

\begin{corollary}
Let $m^2$ be a strict subpandigital or a strict subpenholodigital square in base $b$.
If $b$ is an even squarefree number, then
$m\equiv 0 \pmod{b-1}$.
If $b$ is an odd square free number, then
$m\equiv (b-1)/2 \pmod{b-1}$.
\end{corollary}

For oblong numbers, we have the following results:

\begin{theorem} \label{thm:oblong:sub}
Let $B_b$ be as defined in Theorem \ref{thm:oblong}.
If $n(n+1)$ is a strict subpandigital or a strict subpenholodigital oblong number, then $n\equiv m \pmod{b-1}$ for some $m\in B_b$. 
\end{theorem}
\begin{corollary} \label{cor:oblong:sub}
If $b\equiv 3 \pmod{4}$, then there are no strict subpandigital nor strict subpenholodigital oblong numbers in base $b$.
\end{corollary}

Similarly, we conjecture the following:
\begin{conjecture}
Suppose $b>4$.
A strict subpandigital oblong number and a strict subpenholodigital oblong number in base $b$ exists
if and only if $b\not\equiv 3\pmod{4}$.
\end{conjecture}

The roots of the smallest subpandigital and subpenholodigital oblong number in base $b$ are listed in OEIS sequences  \href{https://oeis.org/A382054}{A382054} and  \href{https://oeis.org/A382055}{A382055} respectively.

For primes, there is an analog to Theorems \ref{thm:prime}-\ref{thm:prime2}:

\begin{theorem} \label{thm:prime:sub}
Let $b>3$. A subpandigital prime in base $b$ must be larger than or equal to $\frac{b^{b-1}-b}{(b-1)^2}+b^{b-1}$ and any subpenholodigital prime must be larger than or equal to
$\frac{b^{b-1}-b}{(b-1)^2}+b^{b-2}$.
\end{theorem}
\begin{proof}
Follows from the fact that $\frac{b^{b-1}-b}{(b-1)^2}+b^{b-1}$ can be written as $10123...(b-2)$ in base $b$ and
$\frac{b^{b-1}-b}{(b-1)^2}+b^{b-2}$ can be written as $1123...(b-2)$ in base $b$. 
\end{proof}

Similarly, these lower bounds can be improved for bases of the form $b=4k+3$.

\begin{theorem} \label{thm:prime2:sub}
If $b = 4k+3$ for $k>0$, then the smallest subpandigital prime in base $b$ is larger than or equal to $\frac{b^{b-1}-b}{(b-1)^2}+b^{b-1}+b^{b-3}$ and the smallest subpenholodigital prime in base $b$ is larger than or equal to $\frac{b^{b-1}-b}{(b-1)^2}+b^{b-2}+b^{b-3}$.
\end{theorem} 

\begin{proof}
For a subpandigital prime, if $b=4k+3$, then $(b-2)(b-1)/2+1 = 2(4k^2+3k+1)$  and $b-1$ are both even. Thus if $s_b(n) =  (b-2)(b-1)/2+1$, then $n>2$ is even and thus not prime. Thus $s_b(n)\geq (b-2)(b-1)/2+2$ and thus
$n$ is larger than or equal to $10223....(n-2)$ in base $b$ which is equal to $\frac{b^{b-1}-b}{(b-1)^2}+b^{b-1}+b^{b-3}$ 

Similarly, for a subpenholodigital prime, $n$ is larger than or equal to $1223....(n-2)$ in base $b$ which is equal to $\frac{b^{b-1}-b}{(b-1)^2}+b^{b-2}+b^{b-3}$.
\end{proof}

Table \ref{tbl:subpprime} shows the smallest subpandigital primes (OEIS sequence  \href{https://oeis.org/A371511}{A371511}) and subpenholodigital primes (OEIS sequence \href{https://oeis.org/A371512}{A371512}) for various bases. This table suggests that, similar to Section \ref{sec:prime}, the following conjecture:
\begin{conjecture} For $b>4$ the smallest subpandigital prime or smallest subpenholodigital prime $n$ satisfy $s_b(n) = (b-2)(b-1)/2+2$ for $b$ of the form $4k+3$ and satisfy $s_b(n) = (b-2)(b-1)/2+1$ otherwise. 
\end{conjecture}

\begin{landscape}
\begin{table}
\begin{center}
\begin{tabular}{|r || r r | r r|}
\hline
base $b$ & smallest subpandigital prime & written in base $b$ & smallest subpenholodigital prime & written in base $b$\\
\hline\hline
3 & 3 & 10 & 13 & 111 \\
4 & 73 & 1021 & 37 & 211 \\
5 & 683 & 10213 & 163 & 1123 \\
6 & 8521 & 103241 & 1861 & 12341 \\
7 & 123323 & 1022354 & 22481 & 122354 \\
8 & 2140069 & 10123645 & 304949 & 1123465 \\
9 & 43720693 & 101236457 & 5455573 & 11234567 \\
10 & 1012356487 & 1012356487 & 112345687 & 112345687 \\
11 & 26411157737 & 10223456798 & 2831681057 & 1223456987 \\
12 & 749149003087 & 10123459a867 & 68057976031 & 1123458a967 \\
13 & 23459877380431 & 1012345678a9b & 1953952652167 & 112345678ba9 \\
14 & 798411310382011 & 1012345678c9ab & 61390449569437 & 11234567a8bc9 \\
15 & 29471615863458281 & 1022345678a9cdb & 2224884906436873 & 122345678acb9d \\
16 & 1158045600182881261 & 10123456789acbed & 77181689614101181 & 1123456789ceabd \\
17 & 48851274656431280857 & 10123456789acdebf & 3052505832274232281 & 1123456789acebfd \\
18 & 2193475267557861578041 & 10123456789abcefgd & 129003238915759600789 & 1123456789abfcegd \\
19 & 104737172422274885174411 & 10223456789abcedfhg & 6090208982148446231753 & 1223456789abchfedg \\
20 & 5257403213296398892278377 & 10123456789abcdgefih & 276667213296398892309917 & 1123456789abcdgiefh \\
\hline
\end{tabular}
\end{center}
\caption{Smallest subpandigital and subpenholodigital primes.}
\label{tbl:subpprime}
\end{table}
\end{landscape}

\section{Conclusions}
We study properties of pandigital, penholodigital, subpandigital and subpenholodigital numbers in various number bases and give conditions and bounds on when they intersect with the set of prime numbers, the set of square numbers and the set of oblong numbers.

\end{document}